\documentclass[hidelinks,11pt]{article}
\usepackage{amsmath,amssymb,amsfonts,amsthm,mathrsfs,relsize,mathtools,graphicx,float,hyperref,yhmath,euscript,hyperref,cleveref,thmtools,thm-restate,enumitem}
\usepackage[title]{appendix}
\usepackage{titlesec}
\titlelabel{\thetitle.\,\,}
\allowdisplaybreaks
\setlist[itemize]{noitemsep,topsep=0pt}

\let\svthefootnote\thefootnote
\newcommand\blankfootnote[1]{%
\let\thefootnote\relax\footnotetext{#1}%
\let\thefootnote\svthefootnote%
}

\theoremstyle{plain}
\newtheorem{theorem}{Theorem}[section]
\newtheorem{lemma}[theorem]{Lemma}

\newtheorem{corollary}[theorem]{Corollary}

\newtheorem{observation}[theorem]{Observation}

\theoremstyle{definition}

\newtheorem{remark}[theorem]{Remark}

\DeclareMathAlphabet{\mathbbmsl}{U}{bbm}{m}{sl}

\DeclareMathAlphabet{\mathpzc}{OT1}{pzc}{m}{it}
\DeclareMathAlphabet{\mathsfit}{T1}{\sfdefault}{\mddefault}{\sldefault}\SetMathAlphabet{\mathsfit}{bold}{T1}{\sfdefault}{\bfdefault}{\sldefault}

\begin{document}

\title{\vspace{-1cm}{\textbf{Weak saturation numbers in random graphs}}\\[6mm]}

\author{
\hspace{-5mm}Olga  Kalinichenko$^{^{1, a}}$  \, \,    Meysam Miralaei$^{^{2, b}}$  \,   \,  Ali Mohammadian$^{^{3,  c, d}}$  \, \,    Behruz Tayfeh-Rezaie$^{^{2, c}}$\\[4mm]
$^{^1}$Moscow Institute of Physics and Technology, \\  Dolgoprudny 141700, Russia   \\[2mm]
$^{^2}$School of Mathematics,    Institute for Research in
Fundamental  Sciences (IPM),  \\  P.O. Box 19395-5746, Tehran, Iran \\[2mm]
$^{^3}$School of Mathematical Sciences,    Anhui University, \\
Hefei 230601,  Anhui,    China \\[2mm]
\hspace{-5mm}\href{mailto:kalinichenko.oi@phystech.edu}{kalinichenko.oi@phystech.edu}  \,  \,   \href{mailto:m.miralaei@ipm.ir}{m.miralaei@ipm.ir} \, \,     \href{mailto:ali\_m@ahu.edu.cn}{ali\_m@ahu.edu.cn} \,  \,   \href{mailto:tayfeh-r@ipm.ir}{tayfeh-r@ipm.ir}\\[6mm]}

\blankfootnote{\hspace*{-6mm}$^{^a}$Partially supported by RFBR under grant number	20-51-56017.\\
$^{^b}$Partially  supported by a grant from IPM.\\
$^{^c}$Partially  supported by Iran  National  Science Foundation   under project number  99003814.\\
$^{^d}$Partially  supported   by the    Natural Science Foundation of Anhui Province  with  grant identifier 2008085MA03 and by the National Natural Science Foundation of China with  grant number 12171002.}

\date{}

\maketitle

\sloppy

\begin{abstract}
For  two given graphs $G$ and $F$, a graph  $ H$ is said to be weakly $ (G, F) $-saturated
if $H$  is a spanning subgraph of $ G$ which has     no   copy of  $F$  as a   subgraph and one  can add all edges in $ E(G)\setminus E(H)$ to $ H$ in some order so  that a new copy of $F$  is created at each step.
The  weak saturation number   $ \mathrm{ wsat}(G, F)$ is the
minimum number of edges of  a weakly $(G, F)$-saturated graph.
In this paper, we deal with the relation between $ \mathrm{ wsat}(\mathbbmsl{G}(n,p), F)$ and $ \mathrm{wsat}(K_n, F)$, where $\mathbbmsl{G}(n,p)$ denotes the Erd\H{o}s--R\'enyi random graph and $ K_n$ denotes the complete graph on $ n$ vertices. For  every    graph $ F$ and  constant $ p$, we prove  that
$ \mathrm{ wsat}( \mathbbmsl{G}(n,p),F)= \mathrm{ wsat}(K_n,F)(1+o(1))$ with high probability.
Also,  for some graphs $ F$ including complete graphs, complete bipartite graphs, and  connected graphs with minimum degree $ 1$ or $ 2$,  it is shown that there exists  an $ \varepsilon(F)>0$ such that, for any $ p\geqslant n^{-\varepsilon(F)}\log n$,   $ \mathrm{ wsat}( \mathbbmsl{G}(n,p),F)= \mathrm{ wsat}(K_n,F)$  with high probability. \\[-1mm]


\end{abstract}

\section{Introduction}

All graphs throughout this paper are assumed to be  finite, undirected, and without loops or multiple edges.  The vertex set of a graph $G$ is denoted by  $V(G)$ and the
edge set of   $G$ is denoted by  $E(G)$.
For two  given graphs $G$ and $F$, a graph $H$ is said to be   {\sl weakly $(G, F)$-saturated}  if $H$ is a spanning subgraph of $ G$ which has no copy of  $F$   as a  subgraph and  there is an ordering $e_1, e_2, \dots$  of edges in $E(G)\setminus E(H)$ such that for $i=1, 2, \ldots$ the spanning subgraph of $G$  with the edge set  $E(H)\cup\{e_1, \ldots, e_{i}\}$ has a  copy    of   $F$ containing   $e_i$.
The minimum number of edges in a weakly  $(G, F)$-saturated graph is called the {\sl weak saturation number}  of $F$ in $G$ and  is denoted by $ \mathrm{ wsat}(G, F)$.
Let $ K_n $ be the complete   graph on   $n$ vertices and $ K_{s,t} $ be the complete bipartite graph with parts of sizes $ s $ and $ t $.
For the purpose of simplification,
a     weakly $(K_n, F)$-saturated graph  is called      weakly $F$-saturated if  there is no danger of ambiguity and moreover,
$\mathrm{ wsat}(K_n, F)$ is written as   $ \mathrm{ wsat}(n, F)$.

The notion of  weak saturation  was  initially introduced by Bollob\'{a}s \cite{Bella} in 1968.
Weak saturation is closely related
to the so-called  ‘graph bootstrap percolation’   which  was    introduced for the first time in  \cite{Balogh}.
It is worth mentioning that the study of any  extremal parameter    is an important task in    graph theory and usually receives  a great deal of  attention.
Determining the exact value of  $ \mathrm{ wsat}(n, F)$   for a given graph $F$ is  often  quite difficult.
Although  the  weak  saturation number has   been   studied for a long time, related   literature is still poor.
Lov\'{a}sz \cite{Lovasz} established  that
$$\mathrm{ wsat}(n, K_s)=(s-2)n-\binom{s-1}{2}$$
when    $n\geqslant s\geqslant2$,    settling a conjecture of Bollob\'{a}s  \cite{Bella}.
Kalai \cite{Kalai2}   proved  that   $$\mathrm{wsat}(n, K_{t, t})=(t-1)n - \binom{t-1}{2}$$  if   $n\geqslant4t-4$.
An alternative proof for the  result is given by Kronenberg, Martins,  and Morrison   \cite{Kronenberg}  for every  $n\geqslant3t-3$. They also established  that,    for every   $t>s$ and sufficiently large $ n$,
$$(s - 1)(n - t + 1) + {t \choose 2} \leqslant \mathrm{wsat}(n, K_{s,t}) \leqslant (s - 1)(n - s) + {t \choose 2}.$$
Miralaei, Mohammadian, and Tayfeh-Rezaie \cite{MMT} determined the exact value of $ \mathrm{ wsat}(n, K_{2,t})$. They found that,    for every   $t \geqslant 3$ and $n \geqslant t + 2$,
\begin{eqnarray*}
\mathrm{ wsat}(n, K_{2,t})=\left\{
\begin{array}{ll}\vspace{-4.6mm}&\\
n - 1 + {t \choose 2} &  \quad   \text{if $ t$ is even and $ n\leqslant 2t-2$},\\ \vspace{-2mm} \\
n - 2 + {t \choose 2} & \quad   \text{otherwise}.\\\vspace{-4.6mm}&
\end{array}\right.
\end{eqnarray*}
For  more  results on  weak saturation and related topics, we refer to the survey \cite{Faud.3}.

Random analogues of different parameters in extremal graph theory have  been extensively studied in the literature. These studies often  reveal the behavior of extremal parameters for a typical graph.
Recall that  the Erd\H{o}s--R\'enyi random graph model $\mathbbmsl{G}(n, p)$ is the
probability space of all graphs on a fixed vertex set of size $n$ where every two  distinct    vertices  are adjacent   independently with probability $p$. Also,    recall that the  notion `with high probability'  is used whenever   an event  occurs in     $\mathbbmsl{G}(n, p)$  with a  probability approaching $1$ as $n$ goes to infinity.

The study of the
weak saturation problem in random graphs  was initiated by Kor\'andi and   Sudakov \cite{kor}. They   proved that,  for every constant $p\in(0, 1)$ and    integer  $s\geqslant3$,
$$\mathrm{wsat}\big(\mathbbmsl{G}(n, p), K_s\big)=\mathrm{wsat}(n, K_s)$$
with high probability.
We will sometimes use the notion  `stability' for the  graph $F$ if  $ \mathrm{ wsat}(\mathbbmsl{G}(n, p), F)= \mathrm{ wsat}(n, F)$ with high probability.
Bidgoli, Mohammadian, Tayfeh-Rezaie, and Zhukovskii \cite{BMTZ} established the existence of a threshold function for the property $ \mathrm{ wsat}(\mathbbmsl{G}(n, p), K_s)= \mathrm{ wsat}(n, K_s)$ and provided the following upper and lower bounds on the function for any $ s\geqslant 3$.

\begin{itemize}[leftmargin=*]
\item  There exists    a constant $c_s$ such that, if $p \leqslant c_s n^{-2/(s+1)}\log ^{2/((s-2)(s+1))}n$, then  $\mathrm{wsat}(\mathbbmsl{G}(n, p), K_s) \neq \mathrm{wsat}(n, K_s)$ with high probability.
\item  If $p\geqslant n^{-1/(2s-3)}\log^{(s-1)/(2s-3)} n$, then $\mathrm{wsat}(\mathbbmsl{G}(n, p), K_s) = \mathrm{wsat}(n, K_s)$ with high probability.
\end{itemize}

Borowiecki and  Sidorowicz \cite{Borowiecki} proved that  $$\mathrm{wsat}(n, K_{1,t})=\binom{t}{2}$$ provided    $n\geqslant t+1$.   A short proof of the  result is given in  \cite{FGJ}.
Kalinichenko and  Zhukovskii  \cite{Kalinichenko} investigated $\mathrm{wsat}(\mathbbmsl{G}(n, p), K_{1, t})$ and  provided the following bounds for $ t\geqslant 3$.

\begin{itemize}[leftmargin=*]
\item  There exists   a constant $c_t$ such that, if $n^{-2}\ll p\leqslant c_t n^{-1/(t-1)}\log ^{-(t-2)/(t-1)} n$,
then  $\mathrm{wsat}(\mathbbmsl{G}(n, p), K_{1, t}) \neq \mathrm{wsat}(n, K_{1, t})$ with high probability.
\item  There exists   a constant $d_t$ such that, if $p\geqslant d_t n^{-1/(t-1)}\log ^{-(t-2)/(t-1)} n$, then $\mathrm{wsat}(\mathbbmsl{G}(n, p), K_{1, t}) = \mathrm{wsat}(n, K_{1, t})$ with high probability.
\end{itemize}

For a graph  $G$ and  a  subset   $X$ of $V(G)$,  denote by $N_G(X)$ the set of vertices of $G$  which are  adjacent to all vertices in   $X$ and set  $ N_G[X]= X\cup N_G(X)$.
For the sake of convenience, $N_G(v_1, \ldots, v_k)$ is written instead of $N_G(\{v_1, \ldots, v_k\})$.
For a vertex $ v$ of $ G$, we define the   {\sl degree}  of $v$    as $|N_G(v)|$ and denote it by $ d_G(v)$.
The maximum and   minimum degrees of the  vertices of $ G$ are denoted by $ \mathnormal{\Delta}(G)$ and  $ \delta(G)$, respectively.
Furthermore, denote by  $e(\mathbbmsl{G}(n,p))$    the random variable counting the edges in  $ \mathbbmsl{G}(n,p) $.

Kalinichenko and  Zhukovskii  \cite{Kalinichenko}    studied  sufficient conditions  on  weakly $(K_n, F)$-saturated graphs  such  that the equality     $ \mathrm{ wsat}(\mathbbmsl{G}(n, p), F)= \mathrm{ wsat}(n, F)$ holds with high probability. They proved the following theorem.

\begin{theorem} [\cite{Kalinichenko}]\label{thm:zhukovskii}
Let $F$ be a graph with   $ \delta(F) \geqslant 1$. Also, let $p \in (0,1)$ and $c\geqslant\delta(F)-1$ be constants. For every positive integer $n$, assume that there exists a   weakly $(K_n, F)$-saturated graph $H_n$   containing a set of vertices $S_n$ with $|S_n| \leqslant c$ such that each  vertex from $V(K_n) \setminus S_n$ is adjacent to  at least $\delta(F) - 1$ vertices in  $S_n$.  Then,  there exists a   weakly $(\mathbbmsl{G}(n, p), F)$-saturated graph with  $\min\{e(\mathbbmsl{G}(n,p)), |E(H_n)|\}$ edges with high probability.
\end{theorem}

The following corollary  immediately follows from \Cref{thm:zhukovskii} and a result due to Spencer \cite{Spencer}.

\begin{corollary}[\cite{Kalinichenko}]\label{corol_transfer}
Let $F$ be a graph with   $ \delta(F) \geqslant 1$ and let
$p\in(0,1)$ be constant.
For every positive integer $n$, assume that  there  exists     a  weakly $(K_n, F)$-saturated graph $H_n$  with $|E(H_n)|= \mathrm{wsat}(n, F) $   satisfying the property described in \Cref{thm:zhukovskii}. Then,
$\mathrm{wsat}(\mathbbmsl{G}(n, p), F) = \mathrm{wsat}(n, F)$ with high probability.
\end{corollary}

It has been verified in  \cite{Kalinichenko}  that \Cref{corol_transfer} implies the stability for some graphs $F$ such as complete graphs and complete bipartite graphs.

In this paper, we continue to explore the relation between
$\mathrm{ wsat}(\mathbbmsl{G}(n, p), F)$ and $ \mathrm{ wsat}(n, F)$.  We find the asymptotic behavior  of
$\mathrm{ wsat}(\mathbbmsl{G}(n, p), F)$ and sometimes its exact value compared to $ \mathrm{ wsat}(n,F) $.
Regarding the asymptotic behavior, we prove the following theorem in \Cref{S.2}.

\begin{restatable}{theorem}{asympstab}\label{th:asymp_stab}
Let  $ F$ be a graph with   $ \delta(F) \geqslant 1$   and let $p\in (0,1) $ be  constant. Then, 	$ \mathrm{ wsat}( \mathbbmsl{G}(n,p),F)= \mathrm{ wsat}(n,F)(1+o(1))$ with high probability.
\end{restatable}

Notice   that addition of isolated vertices to the graph $F$ does not change the weak saturation number.
Subsequently,  we only consider graphs $ F$  with   $ \delta(F) \geqslant 1$.
Especially, \Cref{th:asymp_stab} holds for all graphs $F$.
In \Cref{S.3}, we   present a sufficient condition on $\mathrm{ wsat}(n, F)$ for which $ \mathrm{ wsat}( \mathbbmsl{G}(n,p),F)= \mathrm{ wsat}(n,F)$ with high probability.

\begin{restatable}{theorem}{GnpKn}\label{thm:Gnp=Kn}
Let $ F$ be a graph  with   $ \delta(F) \geqslant 1$ and let $ \mathrm{ wsat}(n, F)\geqslant (\delta(F)-1)n+D $ for all    $ n $, where $ D$ is a constant depending on $ F $. Then,
there exist    a positive integer $ k$ and a constant $d_F$  such that $ \mathrm{ wsat}(n, F)= (\delta(F)-1)n+d_F $  and,   for any  $ p\geqslant n^{-1/(2k+3)}\log n$,
$\mathrm{ wsat}( \mathbbmsl{G}(n,p), F)= \mathrm{ wsat}(n, F) $  with high probability.
\end{restatable}

For some graphs $ F$ including complete graphs, complete bipartite graphs, and    connected graphs with minimum degree $ 1$ or $ 2$,  \Cref{thm:Gnp=Kn} shows that there is an $ \varepsilon(F)>0$ such that, for any $ p\geqslant n^{-\varepsilon(F)}\log n$,    $ \mathrm{ wsat}( \mathbbmsl{G}(n,p),F)= \mathrm{ wsat}(n,F)$   with high probability,    see \Cref{rem:stability1}.
Note  that \Cref{thm:Gnp=Kn} is a generalization of \Cref{thm:zhukovskii} and \Cref{corol_transfer} with simpler arguments. Actually, a weakly $(K_n, F)$-saturated graph having the structure  described in \Cref{thm:zhukovskii} has $(\delta(F)-1)n + D$ edges for some constant $D$ depending on $F$. Our proof does not require a specific structure of the weakly $(K_n, F)$-saturated graph, only the number of edges in it is needed. This makes it easier to construct arguments and apply the result not only to complete graphs and complete bipartite graphs which can be also easily done with \Cref{thm:zhukovskii}, but also to some other types of graphs.

Finally, we  consider the case $p = o(1)$ in \Cref{S.4} where  we present a condition which implies that $\mathrm{ wsat} (\mathbbmsl{G}(n,p), F) = e(\mathbbmsl{G}(n,p))(1 + o(1))$ with high probability.

Let us fix here more  notation and terminology   that we use in the rest of the  paper.
Let $G$ be a graph.
For a  vertex $ v$ of $ G$, we denote by $G-v$  the   graph obtained from  $G$ by removing   $v$   and all of its incident edges.
For two  nonadjacent vertices  $u, v$    of $ G$, let $G+uv$  denote the   graph obtained from  $G$  by adding an     edge   between   $u$  and $ v$.
For a  subset   $X$ of $V(G)$, we denote the induced subgraph of $G$ on $X$  by $G[X]$.
For two disjoint subsets $ S$ and $ T$ of $ V(G)$,   denote by $ E_{G}(S, T)$ the set of all edges with an endpoint in both $ S$ and $ T$.
For the purpose of simplicity, $E_G(v, T)$ is written instead of $E_G(\{v\}, T)$.
For a positive integer $d$, the {\sl $d$-th power}  of    $G$,   denoted by $G^d$,    is the graph  with    vertex set $V(G)$  and  two vertices $x, y$ are adjacent in $G^d$ if and only if the distance between  $x, y$  in  $G$  is at most $d$.

\section{Asymptotic stability}\label{S.2}

In this section, we prove that
$\mathrm{wsat}(\mathbbmsl{G}(n,p),F)= \mathrm{ wsat}(n,F)(1+o(1))$ with high probability for any constant
$p$  and any graph $F$.
To proceed, we need to recall  the following result on  the appearance of the   powers  of a Hamiltonian  cycle in $ \mathbbmsl{G}(n,p)$.

\begin{theorem}[\cite{kahn, Riordan}]\label{thm:Hamilton}
If $k\geqslant 2$ and  $p\gg n^{-1/k}$, then $\mathbbmsl{G}(n,p)$ contains the $k$-th power of a Hamiltonian cycle with high probability.
\end{theorem}

\begin{lemma}\label{lem:w(n)}
Let $ p=n^{-o(1)} $. Then, there is a function $ w(n) $ tending to infinity as $ n\rightarrow \infty $ such that the vertex set of $  \mathbbmsl{G}(n, p) $ can be partitioned into  cliques of size at least $ w(n) $  with high probability.
\end{lemma}

\begin{proof}
Let $ \EuScript{H}_d$ be the event that $ \mathbbmsl{G}(n, p) $ contains the $d$-th power of a Hamiltonian cycle.
Since $ p=n^{-o(1)}$,  \Cref{thm:Hamilton} implies that, for any  $ \varepsilon >0$ and any  integer $k\geqslant 2 $, there is a minimum integer $ N(k, \varepsilon)$ such that $ \mathbbmsl{P}[ \EuScript{H}_k]\geqslant 1-\varepsilon$ for any $ n\geqslant N(k, \varepsilon)$.
Note that $ N(k, \varepsilon)\geqslant k+1$ and $ N(k_1, \varepsilon_1)\geqslant N(k_2, \varepsilon_2)$ if $ k_1\geqslant k_2$ and $ \varepsilon_1 \leqslant \varepsilon_2$.
Let $ m_1=2$ and for any $ k \geqslant 2$, define
$$m_{k}=\min\left\{\ell \, \left| \, \ell >m_{k-1}  \text{ and }   N\left(\ell, \frac{1}{k}\right)>N\left(m_{k-1}, \frac{1}{k-1}\right)\right.\right\}.$$
Clearly,  $\{N(m_k, \tfrac{1}{k})\}_{k=1}^{\infty} $ is strictly increasing.
Now, define the function $ w$ as $$ w(n)=m_k  \,  \text{ for any }  \, k\geqslant 1 \, \text{ and } \, n\in\left[N\left(m_k, \frac{1}{k}\right), N\left(m_{k+1}, \frac{1}{k+1}\right)\right).$$
Obviously,    $ w(n)\rightarrow \infty$ as $ n\rightarrow \infty$. We show that
$\mathbbmsl{P}[ \EuScript{H}_{w(n)}]\geqslant 1-\tfrac{1}{k} $ for every $ k$ and $ n\geqslant N(m_k, \tfrac{1}{k})$. To see this, fix $ k$ and assume that $ N(m_{\ell}, \tfrac{1}{\ell}) \leqslant n < N(m_{\ell+1}, \tfrac{1}{\ell+1})$ for some $ \ell \geqslant k$.  Then,
$$\mathbbmsl{P}\Big[ \EuScript{H}_{w(n)}\Big] = \mathbbmsl{P}\Big[ \EuScript{H}_{m_{\ell}}\Big]\geqslant 1-\frac{1}{\ell}\geqslant 1-\frac{1}{k}.$$
Thus, $ \mathbbmsl{G}(n, p) $ contains the $w(n)$-th power of a Hamiltonian cycle with high probability, say $ C $.
Note that every $ w(n)+1$ consecutive vertices of $ C $ form a clique in $ \mathbbmsl{G}(n, p) $. Now, by partitioning the vertices of $ \mathbbmsl{G}(n, p) $ so that each part  consists of at least $ w(n)/2$ consecutive vertices  of $ C$, the result follows.
\end{proof}

The following lemma is obtained from Theorem 2 in \cite{Spencer}.

\begin{lemma}[\cite{Spencer}]\label{lem:clique}
For any integer $ s\geqslant 3  $, there exists a constant  $ c $ such that if
$ p\geqslant cn^{-2/(s+1)}\log ^{{2}/((s-2)(s+1))}n$, then  every two vertices of $  \mathbbmsl{G}(n,p) $ have a clique of size $ s-2 $ in their common neighbors with high probability. 	
\end{lemma}

The following corollary immediately follows from   \Cref{lem:clique}.

\begin{corollary}\label{cor:wsatwsat}
Let $ F$ be a graph and let    $ s=|V(F)|\geqslant 3 $  and     $\delta(F)\geqslant 1 $. There is a constant  $ c $ such that if
$p\geqslant cn^{-2/(s+1)}\log ^{{2}/((s-2)(s+1))}n$, then
$\mathrm{ wsat}( \mathbbmsl{G}(n, p), F) \geqslant \mathrm{ wsat}(n, F) $ with high probability.
\end{corollary}

\begin{proof}
Let $ H $ be a weakly $ (\mathbbmsl{G}(n, p), F)$-saturated graph with minimum possible number of edges. Using   \Cref{lem:clique},   $  \mathbbmsl{G}(n, p) $ is weakly $ F$-saturated graph with high probability. This shows that $ H$ is also a weakly $F$-saturated graph, the result follows.
\end{proof}

\begin{remark}
A graph $ G $ is called {\sl balanced}  if $ (|E(H)| -1)/(|V(H)| -2) \leqslant \lambda_G $ for all
proper subgraphs $ H $ of $ G $ with $ |V(H)|\geqslant 3$, where $ \lambda_G= (|E(G)| -2)/(|V(G)| -2)$.
It was proved in \cite{BK} that if $ F $ is a balanced graph and $ p\gg n^{-1/\lambda_F+o(1)} $, then $  \mathbbmsl{G}(n, p) $ is weakly $F$-saturated with high probability,  implying that 	
$\mathrm{ wsat}( \mathbbmsl{G}(n, p), F) \geqslant \mathrm{ wsat}(n, F) $ with high probability.
In general, threshold probability functions  for the property that $  \mathbbmsl{G}(n, p) $ is weakly $F$-saturated are still unknown.
\end{remark}

The following lemma immediately follows from    the  Chernoff  bound \cite[Corollary A.1.2]{Alon.Book} and the union bound.

\begin{lemma}\label{lem:neighbor}
For any positive integer $ k $, there exists  a  constant  $ c$ such that if
$p\geqslant c((\log n)/n)^{1/k}$, then every $ k $-subset of the vertex set of $  \mathbbmsl{G}(n,p) $ has at least $ p^{k}n/2 $ common neighbors with high probability.
\end{lemma}

The following interesting   result, due to   Alon \cite{alon}, will be used in the  next theorem and the subsequent sections.

\begin{theorem}[\cite{alon}]\label{thm:alon}
Let $ F$ be a graph with  $\delta(F)\geqslant 1 $. Then, there exists   a constant $ c_F $ such that $\mathrm{ wsat}(n, F)= (c_F + o(1))n $.
\end{theorem}

We are now in a position to prove  the main result of this section. Recall \Cref{th:asymp_stab}.

\asympstab*

\begin{proof}
It follows from    \Cref{thm:alon} that there exists a constant $c_F$ such that $ \mathrm{ wsat}(n,F)=(c_F+o(1))n$.
By   \Cref{cor:wsatwsat}, it remains to show that $ \mathrm{ wsat}( \mathbbmsl{G}(n,p),F)\leqslant(c_F+o(1))n$ with high probability.

Note that $ \mathrm{ wsat}(G, K_2)=0 $ for every graph $ G$ and so there is nothing to prove  when  $ F=K_2$. Therefore, we may assume that $F$ has $s\geqslant 3$ vertices.
By  \Cref{lem:w(n)}, there is a function $ w(n) $ such that $ w(n) $ goes to infinity when $ n\rightarrow \infty$ and,  with high probability,   the vertex set of
$G \thicksim   \mathbbmsl{G}(n,p)$ admits a partition into $V_1,\ldots,V_m$ such that each $V_i$ is a clique of  size at least $w(n)$.

Fix an $(s-2)$-subset $ S$ in $ V_1$ and let $ i\in\{1, \ldots,  m\}$. If $ |N_G(S)\cap V_i|\geqslant s-1$, then let $ S_i$ be an arbitrary $ (s-1)$-subset of  $ N_G(S)\cap V_i $ and otherwise, let $ S_i$ be an arbitrary $ (s-1)$-subset of  $V_i\setminus S $ containing $N_G(S)\cap V_i$.  Consider the set $ N_G(S \cup S_i)$. Using   \Cref{lem:neighbor}, we have $ |N_G(S \cup S_i)|\geqslant p^{2s-3}n/2$. Since $ p$ is constant and $m\leqslant n/w(n) $, there exists an index $i'$ such that $ | N_G(S \cup S_i) \cap V_{i'}|\geqslant s-2$.
Let $ R_i$ be an arbitrary $ (s-2)$-subset of $ N_G(S \cup S_i) \cap V_{i'} $.

We are going to  introduce  a  weakly $ (G, F) $-saturated graph $ H$  with  $ |E(H)|\leqslant (c_F+o(1))n$. For $i= 1, \ldots,  m$, let $ H_i $ be a weakly $ (G[V_i], F) $-saturated graph with $ \mathrm{ wsat}(|V_i|,F)$ edges. Define $ H$ to be a spanning subgraph of $ G$ with $ E(H)=\bigcup_{i=1}^{m}(E(H_i)\cup E_G(S_i, S\cup R_i) )$.
We have
$$|E(H)|\leqslant	\sum_{i=1}^{m}\Big(\big(c_F+o(1)\big)|V_i|+2(s-1)(s-2)\Big)=\big(c_F+o(1)\big)n.$$
So, it remains to prove that $ H$ is weakly $(G, F)$-saturated. To do this, we show that the edges in $ E(G)\setminus E(H)$ can be added to $ H $ in some order through a weakly $ (G, F)$-saturated process.
We use $ H $ to denote the resulting graph in any step of the process as well. Note that, in each step of the process, every edge $ xy\in E(G)\setminus E(H) $ can be added to $ H$ whenever $ N_H(x, y)$ contains a clique of size   $ s-2 $ in $ H$.
For $ i=1, \ldots,  m$, as $ H_i$ is weakly $ (G[V_i], F)$-saturated, we may add the edges in $E(G[V_i])\setminus E(H_i) $ in an appropriate order.
By doing this, each $ V_i$ becomes a clique in $ H$.

Then, we add all edges in $ E_G(S, N_G(S))\setminus E(H)$ to $ H $ as follows. Let $ x\in N_G(S)\cap V_i$ for some $ i$. If $ x\in S_i$, then $ x$ is already joined to all vertices of $ S$ in $ H$ and there is nothing to prove. So, we may assume that $ x\notin S_i $. In this case, according to the choice of $S_i$, we have  $ S_i \subseteq N_G(S)$ and therefore $ S_i \subseteq N_H(S)$ . Hence, for every $ y\in S $, it follows from $ S_i\subseteq N_H(x, y)$ that the edge $ xy $ can be added to $ H$.

Now, we add all edges in $ E_G(V_i, V_j)\setminus E(H)$ to $ H $ for all distinct indices  $i, j\in \{1, \ldots,  m\}$ as follows. Let $ x\in V_i$ and $ y\in V_j$ with  $ 1\leqslant i<j\leqslant m$ such that $ x,  y$ are not adjacent in $ H$.
Consider the set $ U= N_G(S\cup S_i\cup S_j\cup R_i\cup R_j\cup \{x, y\})$. Using   \Cref{lem:neighbor}, we have $ |U|\geqslant p^{5s-6}n/2$.
Since $ p$ is constant and $m\leqslant n/w(n) $, there exists an index  $ k$ such that $ | U \cap V_k|\geqslant s-2$.
Let $ T$ be an arbitrary $ (s-2)$-subset of $ U \cap V_k $.

We claim that $x$ can be joined to all vertices in $T$ in $H$.
First, we add all remaining edges between $ R_i$ and $ T $ to $ H $. This is possible since both $ R_i$ and  $ T$ are subsets of $ N_G(S) $ and are already joined to $ S$ in $ H$. Next,  we add all remaining edges between $ S_i$ and $ T $ to $ H $. This is possible since both $ S_i$ and  $ T$ are already joined to $R_i$ in $ H$.  Finally, we add all remaining edges between $ x $ and $ T $ to $ H $. This is possible since both $ x$ and  $ T$ are already joined to $ S_i\setminus\{x\} $ in $ H$. Therefore, the claim is proved.
Similarly, $ y$ can be joined to  all vertices in $ T $. Now, we  may add the edge $ xy $ to $ H $, since $ T\subseteq N_H(x, y)$.
The proof is completed.
\end{proof}

\section{Exact stability}\label{S.3}

In this section, we examine the property $ \mathrm{ wsat}(\mathbbmsl{G}(n,p), F)= \mathrm{ wsat}(n,F)$. We present a sufficient condition on $ \mathrm{ wsat}(n, F) $ such that the latter equality holds. The class of graphs satisfying the presented condition includes complete graphs, complete bipartite graphs, graphs with minimum degree $ 1 $, and graphs with minimum degree $ 2 $ which either are  connected or have no cut edge.
First, we recall the following general lower bound on $ \mathrm{ wsat}(n, F)$ which we will later extend it to $ \mathrm{ wsat}(\mathbbmsl{G}(n,p), F)$

\begin{theorem}[\cite{FGJ}]\label{thm:lowbnd}
Let $ F$ be a graph with $ s $ vertices, $ t $ edges, and minimum degree $ \delta\geqslant 1$. Then, for any $ n\geqslant s$,
$$\mathrm{ wsat}(n, F)\geqslant     t-1 + \frac{ (\delta-1)(n-s)}{2}.$$
\end{theorem}

\begin{theorem}
Let $ G$ and $ F$ be two given graphs with $|V(G)|\geqslant |V(F)| $ and $ \delta(F)\geqslant 1$. Then,
$$\mathrm{ wsat}(G, F)\geqslant \min\left\{|E(G)|, |E(F)|-1 + \frac{\min\big\{\delta(G), \delta(F)-1\big\}\big(|V(G)|-|V(F)|\big)}{2}\right\}.$$
\end{theorem}

\begin{proof}
Let $ H $ be a weakly $ (G,F) $-saturated graph with minimum possible number of edges. If $ H=G $, then there is nothing to prove. So, assume that $ H\neq G $. When the first edge in $ E(G)\setminus E(H) $ is added  to $ H $ a copy $ F_1 $ of $ F $ is created. Letting $ H_1=H[V(F_1)] $, we have $ |E(H_1)|\geqslant |E(F)|-1 $.
Also, every vertex $ v\in V(H)\setminus V(H_1) $ must have degree at least
$\min\{\delta(G), \delta(F) -1\} $ in $ H $. To see this, if all edges in $ G $ incident to $ v $  appear in $ H $, then $ d_H(v)=d_G(v)\geqslant\delta(G) $. Otherwise,  when the first edge incident to $ v $ from $ E(G)\setminus E(H) $ is added to $H $, the degree of $ v $ in the resulting graph
must be at least $ \delta(F) $, implying $ d_H(v)\geqslant \delta(F)-1 $. Therefore,
\begin{align*}
|E(H)|&\geqslant |E(H_1)|+|E(H)\setminus E(H_1)|\\
&\geqslant  |E(F)|-1+\frac{\min\big\{\delta(G), \delta(F)-1\big\}\big(|V(G)|-|V(F)|\big)}{2}.
\qedhere
\end{align*}
\end{proof}

We know from    \cite[Theorem 3.4]{A.F} that if $ p\gg \tfrac{\log n}{n} $, then $ \delta( \mathbbmsl{G}(n,p))= np(1+o(1)) $ with high probability. From  this fact, we obtain the following consequence.

\begin{corollary}\label{cor:delta}
Let $ F$ be a graph with $ s $ vertices, $ t $ edges, and minimum degree $ \delta\geqslant 1$. If $ p\gg \tfrac{\log n}{n} $, then
$$\mathrm{ wsat}\big( \mathbbmsl{G}(n,p), F\big)\geqslant t-1 + \frac{ (\delta-1)(n-s)}{2},$$
with high probability.
\end{corollary}

Now, we switch to an upper bound on $\mathrm{wsat}(n, F)$.

\begin{theorem}[\cite{FGJ}]\label{thm:upbnd}
Let $ F$ be a  graph  with $ s$ vertices and minimum degree $ \delta\geqslant 1 $. Then, for every $ n\geqslant m\geqslant s-1 $,
\begin{align}\label{eq:nnn}
\mathrm{ wsat}(n, F)\leqslant (\delta-1)(n-m)+   \mathrm{ wsat}(m, F).
\end{align}
In particular, for every $ n\geqslant s-1 $,
\begin{align}\label{up:nnn}
\mathrm{ wsat}(n, F)\leqslant (\delta-1)n+  \frac{(s-1)(s-2\delta )}{2}.
\end{align}
\end{theorem}

The relation \eqref{eq:nnn} in   \Cref{thm:upbnd}   was proved in \cite{FGJ} for $  m\in\{|V(F)|-1, |V(F)|\} $. Above, we have formulated  a more general statement which can be proved using a similar argument. The inequality \eqref{eq:nnn}  for $ m=|V(F)|-1 $ results in an upper bound  on $ \mathrm{ wsat}(n, F) $ which is given in  \eqref{up:nnn}.

We will also need some knowledge about the presence of small subgraphs in a random graph. First, we recall the following definition.
For any graph $ G$, define
$$m(G)=\max\left\{\left.\frac{|E(H)|}{|V(H)|} \, \right| \,  H  \mbox{ is a subgraph of }  G\right\}.$$
To use later, we  recall the following result which appears in    \cite{A.F} as Theorem 5.3.

\begin{theorem}[\cite{A.F}]\label{thm:threshold}
Let $ G $ be a graph with $|E(G)|\geqslant1$. Then,  $ n^{-1/m(G)}$ is a threshold probability  for the property that $ \mathbbmsl{G}(n,p)$ has a copy of $ G$ as a subgraph.
\end{theorem}

The following lemma was proved in \cite{BMTZ}. We apply it in the next theorem.

\begin{lemma}[\cite{BMTZ}]\label{HAM}
Let  $k\geqslant 2$ be  a fixed integer  and let $p\geqslant n^{-\frac{1}{2k-1}}\log n$. Then, with high probability,  	$\mathbbmsl{G}(n, p)$   has the property that,   for every $k$-subset $S$ of vertices, the induced subgraph on  $N_{G}(S)$ contains the $(k-1)$-th power of a Hamiltonian path.
\end{lemma}

The following theorem gives an upper bound on $ \mathrm{ wsat}( \mathbbmsl{G}(n,p), F) $ which is linear in terms of $ n$. The idea of proof comes from \cite{BMTZ}.

\begin{theorem}\label{thm:upbnd random}
Let $ F$ be a  graph with $ s$ vertices and minimum degree $ \delta\geqslant 1 $. Also, let $ m\geqslant s-1$ and $p\geqslant n^{-1/(2m+3)}\log n$. Then, $$ \mathrm{ wsat}\big( \mathbbmsl{G}(n,p), F\big)\leqslant (\delta-1)(n-m)+   \mathrm{ wsat}(m, F) $$ with high probability. In particular, for $p\geqslant n^{-1/(2s+1)}\log n$,
$$\mathrm{ wsat}\big( \mathbbmsl{G}(n,p), F\big)\leqslant (\delta-1)n+  \frac{(s-1)(s-2\delta )}{2}$$
with high probability.
\end{theorem}

\begin{proof}
As the assertion trivially holds for $ F=K_2 $, we assume that $ s\geqslant 3 $.
By   \Cref{thm:threshold}, we may consider   a clique $\mathnormal{\Omega}$ of size $m$  in $G\thicksim \mathbbmsl{G}(n,p)$. We define a  spanning  subgraph $H$ of $G$ as follows. Let $ H_0 $ be a weakly $ (G[\mathnormal{\Omega}], F) $-saturated graph with $ \mathrm{ wsat}(m,F)$ edges.  The  graph  $H$ contains all edges of $H_0$ and  moreover,  for every $v\in N_G(\mathnormal{\Omega})$, we add  $\delta-1$ arbitrary edges of  $E_G(v, \mathnormal{\Omega})$ to $ H$.
Also, for every $v\in V(G)\setminus N_G[\mathnormal{\Omega}]$, we add   $\delta-1$     edges of  $E_G(v, N_G(v))$   to $H$ as described below.  Using   \Cref{HAM},  the graph $H_v=G[N_G(\{v\}\cup\mathnormal{\Omega})]$ contains  the  $(s-2)$-th power of a Hamiltonian path. Starting from  a beginning vertex, denote the vertices of $H_v$ going in the natural order induced by the Hamiltonian path  by $x^v_1, \ldots, x^v_{h_v}$, where $h_v=|V(H_v)|$. We add the  edges $vx^v_1,  \ldots,  vx^v_{\delta-1}$ to $H$ for any $v\in V(G)\setminus N_G[\mathnormal{\Omega}]$. Since $|E(H)|= (\delta-1)(n-m)+   \mathrm{ wsat}(m, F) $, it suffices to prove that $H$ is a weakly $(G, F)$-saturated graph.
To do this, we show that the edges in $ E(G)\setminus E(H)$ can be added to $ H $ in some order through a weakly $ (G, F)$-saturated process.

As $ H_0$ is weakly $ (G[\mathnormal{\Omega}], F)$-saturated, we may add the edges in $E(G[\mathnormal{\Omega}])\setminus E(H_0) $ to $H$ in an appropriate order.
By doing this, $ \mathnormal{\Omega}$ becomes a clique in $ H$.
Let $ u\in \mathnormal{\Omega} $ and
$ v\in N_G(\mathnormal{\Omega})$ such that $ uv\in E(G)\setminus E(H)$.
Let $ w\in V(F) $ such that $ d_F(w)=\delta(F)$. Since $\mathnormal{\Omega} $ is a clique in $ H$, it contains a copy $ F_0$ of $F-w $ with $ u\in V(F_0)$ and so $ H[V(F_0)\cup \{v\}]+uv$ is a copy of $ F$. This shows that all edges in $ E_G(\mathnormal{\Omega}, N_G(\mathnormal{\Omega}))\setminus E(H)$ may be added to $ H$ simultaneously.	
Now, all edges in $ E(G)\setminus E(H)$ with both endpoints in  $N_G(\mathnormal{\Omega})$ can be added to $ H$,  since they  belong to a copy of  $F$ containing  a $ (s-2)$-subset of $\mathnormal{\Omega}$. Next, for every $v \in V(G)\setminus N_G[\mathnormal{\Omega}]$, we may add to $ H$  the edges $vx^v_{\delta},  \ldots, vx^v_{h_v}$ one by one,  since every such edge belongs to a copy of $F$  containing the previous $s-2$ vertices of the $(s-2)$-th power of the Hamiltonian path. Finally, for each edge $xy\in E(G)\setminus E(H)$ with endpoints in $V(G)\setminus N_G[\mathnormal{\Omega}]$,  by   \Cref{HAM},    $N_G(\{x, y\}\cup\mathnormal{\Omega})$ contains a clique $\mathnormal{\Omega}_{xy}$ of size at least $s-2$
with  $\mathnormal{\Omega}_{xy}\subseteq N_G(\mathnormal{\Omega})$, so $xy$ can be added to $ H$ as well.
\end{proof}

The following consequence is obtained from   \Cref{cor:delta} and   \Cref{thm:upbnd random}.

\begin{corollary}
Let $ F$ be a graph with $ s $ vertices and  $ \delta(F)\geqslant 1$ and let $p\geqslant n^{-1/(2s+1)}\log n$. Then,
$\mathrm{ wsat}( \mathbbmsl{G}(n,p), F) $	is bounded from above by a constant  with high probability if and only if $ \delta(F)=1$.
\end{corollary}

We  are now  ready to prove the main result of this section. Recall \Cref{thm:Gnp=Kn}.

\GnpKn*

\begin{proof}
For simplicity,  let $\delta=\delta(F)$.
By   \Cref{thm:alon}, 	there exist a constant $c_F$ and  a function  $\varphi=o(n)$   such that  $ \mathrm{ wsat}(n, F)=c_Fn+\varphi(n) $  for all $ n$. By the assumption and    \Cref{thm:upbnd}, we have
$$(\delta-1)n+D	\leqslant \mathrm{ wsat}(n, F)\leqslant (\delta-1)n+  \frac{(s-1)(s-2\delta )}{2},$$
where $ s=|V(F)| $.
Since  $ \varphi(n)=o(n) $, the inequality above implies that $ c_F=\delta-1 $ and $ \varphi(n) \geqslant D$ for any $ n$. By   \Cref{thm:upbnd}, for any $ n\geqslant m\geqslant s-1 $,
\begin{align*}
(\delta -1)n+\varphi(n)&= \mathrm{ wsat}(n, F)\\
&\leqslant (\delta-1)(n-m)+   \mathrm{ wsat}(m, F)\\
&=(\delta-1)(n-m)+ \big((\delta -1)m+\varphi(m)\big)\\
&=(\delta -1)n+\varphi(m)
\end{align*}
and so $ \varphi(n)\leqslant \varphi(m) $. This shows that $ \varphi $ is decreasing. As $ \varphi $ is an integer valued function and bounded  from  below, there exists a constant $ d_F $ such that $ \varphi(n)=d_F $ for any  sufficiently large $n$. Assume that  $ k $ is the smallest $ n $ satisfying $ n\geqslant s-1 $ and $ \varphi(n)=d_F $.

Let $ p\geqslant n^{-1/(2k+3)}\log n $. \Cref{cor:wsatwsat} yields that $  \mathrm{ wsat}( \mathbbmsl{G}(n, p), F) \geqslant  \mathrm{ wsat}(n, F) $ with high probability. Moreover,  \Cref{thm:upbnd random} implies that
$$\mathrm{ wsat}\big( \mathbbmsl{G}(n,p), F\big)\leqslant (\delta-1)(n-k)+  \mathrm{ wsat}(k, F),$$
with high probability. Therefore,
\begin{align*}
(\delta -1)n+d_F&= \mathrm{ wsat}(n, F)\\
&\leqslant  \mathrm{ wsat}\big( \mathbbmsl{G}(n, p), F\big)\\
&\leqslant (\delta-1)(n-k)+   \mathrm{ wsat}(k, F)\\
&=(\delta-1)(n-k)+ \big((\delta -1)k+d_F\big)\\
&=(\delta -1)n+d_F,
\end{align*}
and hence $  \mathrm{ wsat}( \mathbbmsl{G}(n,p), F)=(\delta -1)n+d_F= \mathrm{ wsat}(n, F) $ with high probability, as required.
\end{proof}

The following result shows that some graphs  with minimum degree $  2 $ satisfy the assumption of  \Cref{thm:Gnp=Kn}.

\begin{lemma}\label{lem:delta2}
Let $ F $ be a graph with $\delta(F)= 2 $. If $ F$ either is  connected or has no cut edge, then $ \mathrm{ wsat}(n, F)\geqslant n-1$ for any positive integer $ n $.
\end{lemma}

\begin{proof}
Let $ H $ be a weakly $(K_n, F)$-saturated graph with minimum possible number of edges. If $ H $ is connected, then
$\mathrm{ wsat}(n, F)= |E(H)|\geqslant n-1$, as required. So, let $ H $ be  disconnected. Then, the first edge in $ E(K_n)\setminus E(H) $  joining two connected components of $ H $ in a  weakly $ F$-saturated process  is  a cut edge of a copy of $ F$.
Hence, in order to prove the assertion, we may assume that $ F$ is a connected graph with a cut edge.

We prove that $  \mathrm{ wsat}(n, F)\geqslant n$ for any $ n\geqslant 3 $.
To do this,  we show that none of the connected components of $ H $ is a tree. By contradiction, suppose that $H  $ has a connected component $ T $ that is a tree. Assume that $ uv\in E(K_n)\setminus E(H) $ is the first edge in the order of weakly $ F$-saturated process which has an endpoint in $ V(T) $. First, suppose that $ u\in V(T) $ and $ v\notin V(T) $. Let $ F' $ be a copy of $ F $ created by adding $ uv $ during the process. Since the induced subgraph of $ F' $ on $ V(F')\cap V(T) $ is a forest, either it is $ K_1 $ or it has at least two vertices of degree at most $ 1 $. This implies that $ \delta(F)\leqslant1$, a contradiction.  Next, suppose that $ u, v\in V(T)$.
Since $ T $ is a tree and $ F$ is connected, $ F $ has exactly one cycle. As $ \delta(F)=2 $, $ F $ must be a cycle graph, a contradiction.
\end{proof}

\begin{remark}
We give an example of a graph $ F$ with minimum degree $2$ such that $ \mathrm{ wsat}(n, F)<n-1$.
Let $ F=K_3\cup D_3$, where $ D_3$ is a graph consisting of two vertex disjoint triangles which are joined by a single edge. Letting $ n\geqslant 15$ and $n\equiv0 \pmod 3$,  it is easy to
check that $ D_3\cup \tfrac{n-6}{3}P_3$ is  a weakly $F$-saturated graph which implies that  $ \mathrm{ wsat}(n, F)\leqslant \tfrac{2n}{3}+3$.
\end{remark}

\begin{remark}\label{rem:stability1}
It is worth mentioning that the assumption $  \mathrm{ wsat}(n, F)\geqslant (\delta-1)n+D $ for a constant $D$, given in   \Cref{thm:Gnp=Kn}, holds for several graph families. These include complete graphs \cite{Lovasz}, complete bipartite graphs \cite[Proposition 15]{Kronenberg}, graphs with minimum degree $ 1 $  (\Cref{thm:lowbnd}), and graphs with minimum degree $ 2 $ which either are  connected or
have no cut edge (\Cref{lem:delta2}). So,   \Cref{thm:Gnp=Kn}
can be applied for these graph families.
\end{remark}

\section{Small $\boldsymbol{p}$}\label{S.4}

For any  graph $ F$, if $p$ is small enough, then  $ \mathbbmsl{G}(n,p) $ has no copy of $ F$ as a subgraph with high probability. Hence, for such $ p$,   with high probability,  $ \mathbbmsl{G}(n,p) $ is the
unique weakly $(\mathbbmsl{G}(n,p), F)$-saturated graph which means that $ \mathrm{ wsat}(\mathbbmsl{G}(n,p), F)=e(\mathbbmsl{G}(n,p)) $.
In this section, we derive a similar result for slightly bigger $p$.

For any graph $ G$, define
$$\mu(G)=\max\left \{m(G), \frac{|E(G)|-1}{|V(G)|-2}\right\}$$
if $ |V(G)|\neq 2$ and otherwise, define $ \mu(G)$ to be equal to $m(G)$.
Below, we state  a useful observation which can be proved straightforwardly.

\begin{observation}\label{prop:small}
For every two graphs $ G $ and $ F$, let $ X_F(G)$ denote the number of copies of $ F$ in $ G$. Then,
$$|E(G)|-X_F(G)\leqslant \mathrm{ wsat}(G, F)\leqslant |E(G)|.$$
\end{observation}

\begin{theorem}\label{thm:maxDel}
Let $ F$ be a graph with $ \delta(F)\geqslant 1$. Then,  for any $ p\ll n^{-1/\mu(F)}$,
$$\mathrm{ wsat}\big(\mathbbmsl{G}(n,p), F\big)=e\big(\mathbbmsl{G}(n,p)\big)\big(1+o(1)\big)$$ with high probability.
\end{theorem}

\begin{proof}
For  simplicity, let $ s=|V(F)|$, $t=|E(F)| $, $ m=m(F) $, and $ \mu=\mu(F)$.	
If $ \mathnormal{\Delta}(F)=1$, then $ F=tK_2$ and thus    $ e(\mathbbmsl{G}(n,p))=0 $ with high probability using    \Cref{thm:threshold}, so there is nothing to prove. Hence, we assume that $ \mathnormal{\Delta}(F)\geqslant 2$ which gives $ m(F)\geqslant\tfrac{2}{3}$.			
Denote by  $ X_F$   the random variable that counts the number of copies of $ F$ in $ \mathbbmsl{G}(n,p)$.
Let $ p\leqslant n^{-1/\mu}\omega(n)$, where $\omega(n)\rightarrow0$ when $n\rightarrow\infty$.
From    \Cref{prop:small},  $$e\big(\mathbbmsl{G}(n,p)\big)-X_F\leqslant \mathrm{ wsat}\big(\mathbbmsl{G}(n,p), F\big)\leqslant e\big(\mathbbmsl{G}(n,p)\big)$$ and therefore
it is enough to show that $ X_F=o(e(\mathbbmsl{G}(n,p)))$.
If $ \mu \leqslant m$, then it follows from   \Cref{thm:threshold} that $ X_F=0$, we are done.
If $ p\leqslant n^{-7/4}$, then $ p\ll n^{-3/2}\leqslant n^{-1/m}$ and hence $ X_F=0$    using   \Cref{thm:threshold}, again we are done.
So, we may assume that $ \mu=\tfrac{t-1}{s-2} $ and $ p\geqslant n^{-7/4}$.
Since $ p\gg n^{-2}$, it follows from  \cite[Theorem 4.4.4]{Alon.Book} that
$ e(\mathbbmsl{G}(n,p))= n^2p/2(1+o(1))$ with high probability.  We know from   \cite[Lemma 5.1]{A.F} that
$$\mathbbmsl{E}(X_F)=\frac{s!}{|Aut(F)|}\binom{n}{s}p^t,$$
where $ Aut(F)$ is the automorphism group of $ F $.
Also, using the  Markov  bound  \cite[Lemma 22.1]{A.F},
$$\mathbbmsl{P}\left[X_F\geqslant \frac{\mathbbmsl{E}(X_F)}{\sqrt{\omega(n)}}\right]\leqslant \sqrt{\omega(n)}\longrightarrow 0$$
which implies that $ X_F\leqslant \mathbbmsl{E}(X_F)/\sqrt{\omega(n)}$ with high probability.
Now,  since $\mathbbmsl{E}(X_F)\leqslant  n^{s}p^t$,  $e(\mathbbmsl{G}(n,p))\geqslant n^2p/3$ with high probability, and $ \mu=\tfrac{t-1}{s-2}$, we  obtain  that  with high probability
\begin{align*}
X_F&\leqslant \frac{\mathbbmsl{E}(X_F)}{\sqrt{\omega(n)}}\\
&\leqslant   \frac{n^sp^t}{\sqrt{\omega(n)}}\\
&\leqslant \frac{3n^{s-2}p^{t-1}}{\sqrt{\omega(n)}}e\big(\mathbbmsl{G}(n,p)\big)\\
& \leqslant \frac{3n^{s-2}}{\sqrt{\omega(n)}}\left(n^{-\frac{1}{\mu}}\omega(n)\right)^{t-1}e\big(\mathbbmsl{G}(n,p)\big)\\
& \leqslant 3\big(\omega(n)\big)^{t-\frac{3}{2}}e\big(\mathbbmsl{G}(n,p)\big)
\end{align*}
which means  that $ X_F=o(e(\mathbbmsl{G}(n,p)))$ 	as  $ t\geqslant 2$.  The proof is complete.
\end{proof}

\begin{corollary}
Let $ F$ be a graph with $ \mu(F) \geqslant2$  and let $p=\tfrac{2c_F+o(1)}{n}$, where $ c_F$ is introduced in     \Cref{thm:alon}. Then,
$\mathrm{ wsat}(\mathbbmsl{G}(n,p), F)=\mathrm{ wsat}(n, F)(1+o(1))$ with high probability.
\end{corollary}

\begin{proof}
As $p\gg n^{-2}$, it follows from \cite[Theorem 4.4.4]{Alon.Book},   \Cref{thm:alon}, and   \Cref{thm:maxDel} that
\begin{align*}\mathrm{wsat}\big(\mathbbmsl{G}(n,p), F\big)&=e\big(\mathbbmsl{G}(n,p)\big)\big(1+o(1)\big)\\&=\frac{n^2p}{2}\big(1+o(1)\big)\\&=\big(c_F+o(1)\big)n\\&=\mathrm{wsat}(n, F)\big(1+o(1)\big)\end{align*}
with high probability.
\end{proof}

\begin{remark}
It has been proved in \cite{BMTZ} that for any fixed integer $ s\geqslant 3$, there exists  a positive  constant  $ \lambda$ such that
\begin{align}\label{up.Max}
\mathrm{wsat}\big(\mathbbmsl{G}(n, p), K_s\big)\leqslant (s-2)n+\frac{ \log^{\lambda}n}{p^{2s-3}}
\end{align}
with high probability. Let $ F$ be a graph on $ s$ vertices with $ \delta(F)\geqslant 1$. Since $\mathrm{wsat}(\mathbbmsl{G}(n, p), F)\leqslant \mathrm{wsat}(\mathbbmsl{G}(n, p), K_s)$, we find that the upper
bound given in  \eqref{up.Max} also holds for $\mathrm{wsat}(\mathbbmsl{G}(n, p), F) $. Using this and our results in the current  paper, we deduce that the following holds for $\mathrm{wsat}(\mathbbmsl{G}(n, p), F)$ with
high probability.
\begin{itemize}[leftmargin=*]
\item If $ p\ll n^{-1/\mu(F)}$, then  $ \mathrm{ wsat}(\mathbbmsl{G}(n,p), F)=e(\mathbbmsl{G}(n,p))(1+o(1)) $.
\item  If $ n^{-1/\mu(F)}\leqslant p \leqslant n^{-1/(s-1)} $, then $\mathrm{ wsat}(\mathbbmsl{G}(n,p), F)\leqslant e(\mathbbmsl{G}(n,p)) $.
\item  If $ n^{-1/(s-1)}\ll p\leqslant n^{-1/(2s+1)} $, then $\mathrm{ wsat}(\mathbbmsl{G}(n,p), F)\leqslant (s-2)n+p^{-(2s-3)}\log^{\lambda}n $.
\item  If $ p\geqslant n^{-1/(2s+1)}\log n $, then $\mathrm{ wsat}(\mathbbmsl{G}(n,p), F)\leqslant (\delta(F)-1)n+O(1) $.
\item  If $p $ is constant, then	$ \mathrm{ wsat}( \mathbbmsl{G}(n,p),F)= \mathrm{ wsat}(n,F)(1+o(1))$.
\end{itemize}
\end{remark}

\end{document}